\documentclass[12pt, reqno]{amsart}
\usepackage{amsmath, amsthm, amscd, amsfonts, amssymb, graphicx, color}
\usepackage[bookmarksnumbered, colorlinks, plainpages]{hyperref}
\hypersetup{colorlinks=true,linkcolor=red, anchorcolor=green, citecolor=cyan, urlcolor=red, filecolor=magenta, pdftoolbar=true}

\newtheorem{theorem}{Theorem}[section]
\newtheorem{lemma}[theorem]{Lemma}

\newtheorem{cor}[theorem]{Corollary}
\theoremstyle{definition}

\theoremstyle{remark}
\newtheorem{remark}[theorem]{\bf{Remark}}
\numberwithin{equation}{section}
\begin{document}

\title [Annular bounds for the zeros of a polynomial] {   Annular bounds for the zeros of a polynomial from companion matrix   }

\author[P. Bhunia and K. Paul]{Pintu Bhunia and Kallol Paul}

\address{(Bhunia) Department of Mathematics, Jadavpur University, Kolkata 700032, West Bengal, India}
\email{pintubhunia5206@gmail.com; pbhunia.math.rs@jadavpuruniversity.in}

\address{(Paul) Department of Mathematics, Jadavpur University, Kolkata 700032, West Bengal, India}
\email{kalloldada@gmail.com; kallol.paul@jadavpuruniversity.in}

\thanks{First author would like to thank UGC, Govt. of India for the financial support in the form of Senior Research Fellowship}
\thanks{}
\thanks{}



\maketitle

\begin{abstract}
Let $p(z)=z^n+a_{n-1}z^{n-1}+a_{n-2}z^{n-2}+\ldots+a_1z+a_0$  be a complex polynomial with $a_0\neq 0$ and $n\geq 3$. Several new upper bounds for the moduli of the zeros of $p$ are developed. In particular, if $\alpha=\sqrt{\sum_{j=0}^{n-1}|a_j|^2}$ and $z$ is any zero of $p$, then we show that
\begin{eqnarray*}
	|z|^2 &\leq & \cos^2 \frac{\pi}{n+1}+|a_{n-2}|+ \frac{1}{4} \left ( |a_{n-1}|+ { \alpha} \right)^2 +  \frac{1}{2}\sqrt{\alpha^2-|a_{n-1}|^2} +  \frac{1}{2}{\alpha},
\end{eqnarray*}
which is sharper than the Abu-Omar and Kittaneh's bound (see in  \cite[Th. 2.10]{OK}) 
\begin{eqnarray*}
	|z|^2 &\leq & \cos^2 \frac{\pi}{n+1}+ \frac{1}{4} \left ( |a_{n-1}|+ { \alpha}\right)^2 +  {\alpha} 
\end{eqnarray*}
if and only if $2|a_{n-2}|< \sqrt{\sum_{j=0}^{n-1}|a_j|^2}-\sqrt{\sum_{j=0}^{n-2}|a_j|^2}. $
The upper bounds obtained here enable us to describe smaller annuli in the complex plane containing all the zeros of $p$.	
\end{abstract}

\section{{Introduction}}
\noindent This paper is concerned with the problem to locate the zeros of a polynomial by using matrix inequalities involving the spectral radius, the numerical radius and the spectral norm. This classical problem attracted many mathematicians over the years beginning with Cauchy.   
This problem is still an enchanting topic to both complex and numerical analysts. One can compute the zeros of a polynomial using the coefficients and their radicals whenever the degree of the polynomial is less than or equal to $4$, however for degree greater than or equal to $5$ this computation is  not always possible. So the study of location  of  the zeros of a polynomial becomes interesting and useful for higher degree poynomials.  The location of the zeros of  polynomials have important applications in many areas of sciences such as signal processing, control theory, communication theory, coding theory and cryptography etc. The Frobenius companion matrix plays an important link between matrix theory and the geometry of polynomials. It has been used to obtain estimations for zeros of polynomials by matrix methods, we refer to some of the recent papers \cite{P15,BBP4,BBP5,KOS} and the references therein. Also, various mathematicians have obtained annular regions containing all the zeros of a polynomial by using classical approach, we refer to \cite{DG1,DG2,Kim} and references therein. 
Suppose that $$p(z)=z^n+a_{n-1}z^{n-1}+a_{n-2}z^{n-2}+\ldots+a_1z+a_0$$  is a complex monic polynomial with coefficients $a_i$ $(i=0,1,\ldots,n-1)$, $a_0\neq 0$ and $n\geq 3$.  Let $C(p)$ be the Frobenius companion matrix $C(p)$ associated with $p,$  which is given by
$$C(p)=\left[\begin{array}{ccccc}
-a_{n-1}&-a_{n-2}&\ldots&-a_1&-a_0\\
1&0&\ldots&0&0\\
0&1&\ldots&0&0\\
\vdots&\vdots& &\vdots&\vdots\\
0&0&\ldots&1&0\\
\end{array}\right]_{n\times n}.$$
It is well-known that the characteristic polynomial of $C(p)$ is $p$ itself, and so the eigenvalues of $C(p)$ are exactly the zeros of $p$, (see \cite[p. 316]{CM}). Many mathematicians have obtained bounds for the moduli of zeros of the polynomial $p$ using the Frobenius companion matrix $C(p)$,  we note few of them in the following. 
Let $\lambda$ be any zero of  $p$. Then 
Linden \cite{L} proved that 
\begin{eqnarray}\label{zero6}
|\lambda| &\leq& \frac{|a_{n-1}|}{n}+\left[ \frac{n-1}{n} \left(n-1+\sum^{n-1}_{j=0}|a_j|^2-\frac{|a_{n-1}|^2}{n}\right)\right]^{\frac{1}{2}}.
\end{eqnarray}
Kittaneh \cite{K} proved that 
\begin{eqnarray}\label{zero5}
|\lambda| &\leq& \frac{1}{2}\left[|a_{n-1}|+1+\sqrt{(|a_{n-1}|-1)^2+4\sqrt{\sum^{n-2}_{j=0}|a_j|^2}}\right].
\end{eqnarray}
Fujii and Kubo \cite{FK} proved that 
\begin{eqnarray}\label{zero4}
|\lambda|&\leq& \cos\frac{\pi}{n+1}+\frac{1}{2}\left[\sqrt{\sum_{j=0}^{n-1}|a_j|^2}+|a_{n-1}|\right].
\end{eqnarray}			
Bhunia et al. \cite{BBP3} proved that 
\begin{eqnarray}\label{zero1}
|\lambda| &\leq& \max\left\{|a_{n-1}|,\cos\frac{\pi}{n}\right\}+\sqrt{\frac{1}{2}\left(1+\sum^n_{j=2}|a_{n-j}|^2\right)}.
\end{eqnarray}
Also, we note some following bounds for the moduli of the zeros of $p$, obtained by classical approach.  \noindent Cauchy \cite{CM}	proved that
\begin{eqnarray}\label{zero2}
|\lambda| &\leq& 1+\max \left\{ |a_0|, |a_1|, \ldots, |a_{n-1}|\right\}.
\end{eqnarray}
Carmichael and Mason \cite{CM} proved that 
\begin{eqnarray}\label{zero3}
|\lambda|\leq \left(1+|a_0|^2+|a_1|^2+\ldots+|a_{n-1}|^2\right)^{\frac{1}{2}}.
\end{eqnarray}
Kim \cite{Kim} proved that if $ p(z) = \sum_{k=0}^na_kz^k\, (a_k \neq 0, 0 \leq k \leq n)$ is a non constant polynomial with complex coefficients, then all the
zeros of $p(z)$ lie in the annulus $ \{z : r_1 \leq |z| \leq r_2\}$, where
$$r_1= \min_{1\leq k\leq n} \left\{\frac{C(n,k)}{2^n-1} \left |\frac{a_0}{a_k}\right|   \right\}^{\frac{1}{k}}$$
and			
$$r_2= \max_{1\leq k\leq n} \left\{\frac{2^n-1}{C(n,k)} \left |\frac{a_{n-k}}{a_n}\right|   \right\}^{\frac{1}{k}},$$
where $C(n,k)=\frac{n!}{k!(n-k)!}$, $0!=1$ are binomial coefficients.\\
Dalal and Govil \cite{DG1} proved that if $ p(z) = \sum_{k=0}^na_kz^k\, (a_k \neq 0, 1 \leq k \leq n)$ is a non constant polynomial with complex coefficients then all the
zeros of $p(z)$ lie in the annulus $ \{z : r_1 \leq |z| \leq r_2\}$, where
$$r_1= \min_{1\leq k\leq n} \left\{\frac{C_{k-1}C_{n-k}}{C_n} \left |\frac{a_0}{a_k}\right|   \right\}^{\frac{1}{k}}$$
and			
$$r_2= \max_{1\leq k\leq n} \left\{\frac{C_n}{C_{k-1}C_{n-k}} \left |\frac{a_{n-k}}{a_n}\right|   \right\}^{\frac{1}{k}},$$
where $C_k=\frac{C(2k,k)}{k+1}$ is the $k$-th Catalan number in which $C(2k,k)$ are the binomial coefficients.

\noindent In this paper, we develope several new bounds for the moduli of the zeros of $p$. These bounds enable us to describe smaller annuli in the complex plane containing all the zeros of $p$.

\section{{Main results}}

\noindent We begin with, noting that for $a\in \mathbb{C}$, $\text{Re}(a)$ and $\text{Im}(a)$ denote the real part and the imaginary part of $a$, respectively, and $\overline{a}$ denotes the complex conjugate of $a$. Let $M_n(\mathbb{C})$ denote the algebra of all $n\times n$ complex matrices.   For $A\in M_n(\mathbb{C}) $, let $W(A)$ denote the numerical range of $A$. Recall that $$W(A)=\left \{{x^*Ax}  : x\in \mathbb{C}^n, {x^*x}=1\right\}.$$ Note that for $x=(x_1,x_2,\ldots,x_n)^t\in \mathbb{C}^n$,   $x^*=(\overline{x_1},\overline{x_2},\ldots,\overline{x_n})$. It is well-known that for every $A\in M_n(\mathbb{C}) $,  $W(A)$ is compact and convex subset of $\mathbb{C}$, see \cite{GR}. For $ A\in M_n(\mathbb{C})$, let $\sigma(A)$, $r(A)$, $w(A)$ and $\|A\|$ denote the spectrum, the spectral radius, the numerical radius and the spectral norm of $A$, respectively. It is easy to observe that $\sigma(A)\subseteq W(A)$. Since $W(A)$ is convex, so $\text{conv}$ $\sigma(A)\subseteq W(A)$ where  $\text{conv}$ $\sigma(A)$ is the convex hull of $\sigma(A)$. Recall that $$r(A)\leq w(A)\leq \|A\|.$$ If $A$ is normal then $\text{conv}$ $\sigma(A)= W(A)$ and $r(A)=w(A)=\|A\|.$ In particular,  if $A$ is Hermitian then $W(A)=\left [\lambda_{\text{min}}(A), \lambda_{\text{max}}(A)\right]$ where $\lambda_{\text{min}}(A)$ and $ \lambda_{\text{max}}(A)$ are the smallest and largest eigenvalue of $A$, respectively. An important property for the numerical radius is  the power inequality, which states that  for every $A\in M_n(\mathbb{C}) $, $w(A^n)\leq w^n(A)$ for $n\in \mathbb{N}$, and equality holds if $A$ is normal. We need the following lemmas to develope the desired bounds.

\begin{lemma}\cite{HD}\label{lem1}
	Let $A=[a_{ij}]_{n\times n}$, $\tilde{A}=[|a_{ij}|]_{n\times n} \in M_n(\mathbb{C}) $. Then, $$\|A\|\leq \|\tilde{A}\|, \,\, w(A)\leq w(\tilde{A}), \,\, r(A)\leq  r(\tilde{A}).$$
\end{lemma}

\begin{lemma}\cite[p.44]{HJ2}\label{lem2}
	Let  $A \in M_n(\mathbb{C})$ be such that $A=[a_{ij}]_{n\times n}$ with $a_{ij}\geq 0.$ Then, 
	$$w(A)=r\left(\left [\frac{ a_{ij}+a_{ji}}{2}\right]_{n\times n}\right).$$
	
\end{lemma}


\begin{lemma}\cite[pp. 8-9]{GR}\label{lem3}
	Let  $$L_n=\left[\begin{array}{cccccc}
	0&\frac{1}{2}&0&\ldots&0&0\\
	\frac{1}{2}&0&\frac{1}{2}&\ldots&0&0\\
	0&\frac{1}{2}&0&\ldots&0&0\\
	\vdots&\vdots& &\vdots&\vdots\\
	0&0&0&\ldots&0&\frac{1}{2}\\
	0&0&0&\ldots&\frac{1}{2}&0\\
	\end{array}\right]_{n\times n}.$$ 
	Then the eigenvalues of $L_n$ are $\lambda_j=\cos \frac{j\pi}{n+1}$ for $j=1,2,\ldots,n.$ Hence, by Lemma \ref{lem2},  $$w\left(\left[\begin{array}{ccccc}
	0&0&\ldots&0&0\\
	1&0&\ldots&0&0\\
	0&1&\ldots&0&0\\
	\vdots&\vdots& &\vdots&\vdots\\
	0&0&\ldots&1&0\\
	\end{array}\right]_{n\times n}\right)=\cos \frac{\pi}{n+1}.$$
\end{lemma}

The fourth lemma is as follows.
\begin{lemma}\label{lem10}
	Let $A\in M_n(\mathbb{C})$ be partitioned as
	\[A=\left[\begin{array}{cc}
	A_{11}&O_{12}\\
	O_{21}&A_{22}
	\end{array}\right],\] where $A_{ij}$ is a matrix of order $n_i\times n_j$ and $O_{ij}$ is the zero matrix of order $n_i\times n_j$ for $i,j=1,2$, and $n_i+n_j=n$ if $i\neq j$. 
	Then, 
	\[w(A)=\max \big \{w(A_{11}),w(A_{22}) \big \}.\]
\end{lemma}

Our first bound for the moduli of the zeros of $p$, is given as:

\begin{theorem}\label{th1}
If $z$ is any zero of $p$, then
\[|z|\leq \cos \frac{\pi}{n}+ \frac{1}{2} \left ( |a_{n-1}|+ \sqrt{(1+|a_{n-2}|)^2+ \sum_{j=0,\,j\neq n-2}^{n-1}|a_j|^2}\right).\]
\end{theorem}

\begin{proof}
First we consider $C(p)=R +S $ where $$R=\left[\begin{array}{ccccc}
	-a_{n-1}&-a_{n-2}&\ldots&-a_1&-a_0\\
	1&0&\ldots&0&0\\
	0&0&\ldots&0&0\\
	\vdots&\vdots& &\vdots&\vdots\\
	0&0&\ldots&0&0\\
	\end{array}\right]_{n\times n} \, \text{and }\, S=\left[\begin{array}{ccccc}
	0&0&\ldots&0&0\\
	0&0&\ldots&0&0\\
	0&1&\ldots&0&0\\
	\vdots&\vdots& &\vdots&\vdots\\
	0&0&\ldots&1&0\\
	\end{array}\right]_{n\times n}. $$
It follows 	from the triangle inequality of the numerical radius is that
	$$w(C(p))\leq w(R)+w(S).$$
From Lemma \ref{lem1}, we have
\begin{eqnarray*}
w(R) &\leq &~~~~w\left(\left[\begin{array}{ccccc}
|a_{n-1}|&|a_{n-2}|&\ldots&|a_1|&|a_0|\\
1&0&\ldots&0&0\\
0&0&\ldots&0&0\\
\vdots&\vdots& &\vdots&\vdots\\
0&0&\ldots&0&0\\
\end{array}\right]_{n\times n}\right)\\
&= &r\left(\left[\begin{array}{ccccc}
|a_{n-1}|&\frac{1}{2}(1+|a_{n-2}|)&\ldots&\frac{1}{2}|a_1|&\frac{1}{2}|a_0|\\
\frac{1}{2}(1+|a_{n-2}|)&0&\ldots&0&0\\
\frac{1}{2}|a_{n-3}|&0&\ldots&0&0\\
\vdots&\vdots& &\vdots&\vdots\\
\frac{1}{2}|a_0|&0&\ldots&0&0\\
\end{array}\right]_{n\times n}\right),\\
& & \mbox{\hspace{6cm}(using Lemma \ref{lem2})}
\end{eqnarray*}
$$ = \frac{1}{2} \left ( |a_{n-1}|+ \sqrt{(1+|a_{n-2}|)^2+ \sum_{j=0,\,j\neq n-2}^{n-1}|a_j|^2}\right).$$
Next by using Lemmas \ref{lem10} and \ref{lem3} we get,
$$w(S)=w\left(\left[\begin{array}{cccc}
0&\ldots&0&0\\
1&\ldots&0&0\\
\vdots& &\vdots&\vdots\\
0&\ldots&1&0\\
\end{array}\right]_{n-1\times n-1}\right)=\cos \frac{\pi}{n}.$$ Therefore, 
\begin{eqnarray*}
w(C(p))&\leq& \cos \frac{\pi}{n}+ \frac{1}{2} \left ( |a_{n-1}|+ \sqrt{(1+|a_{n-2}|)^2+ \sum_{j=0,\,j\neq n-2}^{n-1}|a_j|^2}\right).
\end{eqnarray*}
Since, $|z|\leq r(C(p))\leq w(C(p))$, so \[|z|\leq \cos \frac{\pi}{n}+ \frac{1}{2} \left ( |a_{n-1}|+ \sqrt{(1+|a_{n-2}|)^2+ \sum_{j=0,\,j\neq n-2}^{n-1}|a_j|^2}\right).\]
\end{proof}

Next bound reads as follows.

\begin{theorem}\label{th7}
	If $z$ is any zero of $p$, then
	\begin{eqnarray*}
	|z|^2 &\leq& \cos^2 \frac{\pi}{n}+ \frac{1}{4} \left ( |a_{n-1}|+ \sqrt{(1+|a_{n-2}|)^2+ \sum_{j=0,\,j\neq n-2}^{n-1}|a_j|^2}\right)^2\\
	&&+ \sqrt{\frac{   \left (1+ \sum_{j=0}^{n-1}|a_j|^2 \right)+ \sqrt{\left(1-\sum_{j=0}^{n-1}|a_j|^2\right)^2+4|a_{n-1}|^2 }    }{2}  }.
	\end{eqnarray*}
	
\end{theorem}

\begin{proof}
	Let $R$ and $S$ be as in the proof of Theorem \ref{th1}. Then  $C(p)=R +S $. 
	Following \cite[Lemma 2.9]{OK} we have,
	\[w^2(C(p))=w^2(R+S) \leq w^2(R)+w^2(S)+ \|R\|\|S\|+w(S^*R).\]
	Proceeding as in Theorem \ref{th1}, we get  $$w(R)\leq \frac{1}{2} \left ( |a_{n-1}|+ \sqrt{(1+|a_{n-2}|)^2+ \sum_{j=0,\,j\neq n-2}^{n-1}|a_j|^2}\right)~\mbox{and}~ w(S)=\cos \frac{\pi}{n}.$$
	Simple calculation shows that $\|S\|=1$ and $$\|R\|=\sqrt{\frac{   \left (1+ \sum_{j=0}^{n-1}|a_j|^2 \right)+ \sqrt{\left(1-\sum_{j=0}^{n-1}|a_j|^2\right)^2+4|a_{n-1}|^2 }    }{2}  }.$$ 
	Since, $S^*R=0$ so $w(S^*R)=0.$ Therefore, we have
	\begin{eqnarray*}
		w^2(C(p)) &\leq& \cos^2 \frac{\pi}{n}+ \frac{1}{4} \left ( |a_{n-1}|+ \sqrt{(1+|a_{n-2}|)^2+ \sum_{j=0,\,j\neq n-2}^{n-1}|a_j|^2}\right)^2\\
		&&+ \sqrt{\frac{   \left (1+ \sum_{j=0}^{n-1}|a_j|^2 \right)+ \sqrt{\left(1-\sum_{j=0}^{n-1}|a_j|^2\right)^2+4|a_{n-1}|^2 }    }{2}  }.
	\end{eqnarray*}
	This completes the proof.

	\end{proof}

In the following theorem we obtain another  bound for the zeros.

\begin{theorem}\label{th8}
	If $z$ is any zero of $p$, then
	\begin{eqnarray*}
		|z|^2&\leq& \cos^2 \frac{\pi}{n}+ \frac{1}{4} \left ( |a_{n-1}|+ \sqrt{(1+|a_{n-2}|)^2+ \sum_{j=0,\,j\neq n-2}^{n-1}|a_j|^2}\right)^2\\
		&&+ \frac{1}{2}  \sqrt{(1+|a_{n-4}|)^2+ \sum_{j=0,\,j\neq n-4}^{n-3}|a_j|^2},
	\end{eqnarray*}
 $($$a_{-1}=0$ if $n=3$$)$.	
\end{theorem}

\begin{proof}
	Let $R$ and $S$ be as in the proof of Theorem \ref{th1}. Then  $C(p)=R +S $. 
It is easy to verify that $$r^2(R+S)=r\left((R+S)^2\right)\leq w\left((R+S)^2\right)\leq w^2(R)+w^2(S)+w(RS+SR).$$
Proceeding as Theorem \ref{th1}, we have $w^2(S)=\cos^2 \frac{\pi}{n}$	and $$w^2(R)\leq  \frac{1}{4} \left ( |a_{n-1}|+ \sqrt{(1+|a_{n-2}|)^2+ \sum_{j=0,\,j\neq n-2}^{n-1}|a_j|^2}\right)^2.$$
Now,
	$$RS+SR=\left[\begin{array}{cccccc}
	0&-a_{n-3}&-a_{n-4}&\ldots&-a_0&0\\
	0&0&0&\ldots&0&0\\
	1&0&0&\ldots&0&0\\
	0&0&0&\ldots&0&0\\
	\vdots&\vdots&\vdots& &\vdots&\vdots\\
	0&0&0&\ldots&0&0\\
	\end{array}\right]_{n\times n}.$$
	Therefore, using Lemma \ref{lem1} and Lemma \ref{lem2}, we have
	\begin{eqnarray*}
		w(RS+SR)&\leq& w\left ( \left[\begin{array}{cccccc}
			0&|a_{n-3}|&|a_{n-4}|&\ldots&|a_0|&0\\
			0&0&0&\ldots&0&0\\
			1&0&0&\ldots&0&0\\
			0&0&0&\ldots&0&0\\
			\vdots&\vdots&\vdots& &\vdots&\vdots\\
			0&0&0&\ldots&0&0\\
		\end{array}\right]_{n\times n}\right)\\
		&=&r\left ( \left[\begin{array}{cccccc}
			0&\frac{1}{2}|a_{n-3}|&\frac{1}{2}(1+|a_{n-4}|)&\ldots&\frac{1}{2}|a_0|&0\\
			\frac{1}{2}|a_{n-3}|&0&0&\ldots&0&0\\
			\frac{1}{2}(1+|a_{n-4}|)&0&0&\ldots&0&0\\
			0&0&0&\ldots&0&0\\
			\vdots&\vdots&\vdots& &\vdots&\vdots\\
			\frac{1}{2}|a_0|&0&0&\ldots&0&0\\
			0&0&0&\ldots&0&0\\
		\end{array}\right]_{n\times n}\right) \\
		&=& \frac{1}{2}  \sqrt{(1+|a_{n-4}|)^2+ \sum_{j=0,\,j\neq n-4}^{n-3}|a_j|^2}.
	\end{eqnarray*}
	Hence, 
	\begin{eqnarray*}
		r^2(C(p))&\leq& \cos^2 \frac{\pi}{n}+ \frac{1}{4} \left ( |a_{n-1}|+ \sqrt{(1+|a_{n-2}|)^2+ \sum_{j=0,\,j\neq n-2}^{n-1}|a_j|^2}\right)^2\\
		&&+ \frac{1}{2}  \sqrt{(1+|a_{n-4}|)^2+ \sum_{j=0,\,j\neq n-4}^{n-3}|a_j|^2}.
	\end{eqnarray*}
	This completes the proof.

\end{proof}

For our next result we need the following lemma.

\begin{lemma} \cite{FK}\label{lem4}
	If $c_i\in \mathbb{C}$ for each $i=1,2,\ldots,n,$ then
	$$w\left(\left[\begin{array}{cccc}
	c_1&c_2&\ldots&c_n\\
	0&0&\ldots&0\\	0&&\ldots&0\\
	\vdots&\vdots& &\vdots\\
	0&0&\ldots&0\\
	\end{array}\right]_{n\times n}\right)=\frac{1}{2}\left( |c_1|+ \sqrt{\sum \limits_{j=1}^{n}|c_j|^2} \right).$$
\end{lemma}

\begin{theorem}\label{th2}
	If $z$ is any zero of $p$, then
	\begin{eqnarray*}
	|z|^2&\leq& \cos^2 \frac{\pi}{n}+ \frac{1}{4} \left ( |a_{n-1}|+ \sqrt{(1+|a_{n-2}|)^2+ \sum_{j=0,\,j\neq n-2}^{n-1}|a_j|^2}\right)^2\\
	&& + \frac{1}{4}\sqrt{|a_{n-3}|^2+(1+|a_{n-2}|+|a_{n-4}|)^2+ \sum_{j=0}^{n-4}(|a_{j+1}|+|a_{j-1}|)^2},
	\end{eqnarray*}
where $a_{-1}=0$.
\end{theorem}

\begin{proof}
	It follows from Lemma \ref{lem1} and Lemma \ref{lem2} that 
	\begin{eqnarray*}
	w^2(C(p))
	& \leq& w^2\left(\left[\begin{array}{ccccc}
		|a_{n-1}|&|a_{n-2}|&\ldots&|a_1|&|a_0|\\
		1&0&\ldots&0&0\\
		0&1&\ldots&0&0\\
		\vdots&\vdots& &\vdots&\vdots\\
		0&0&\ldots&1&0\\
	\end{array}\right]_{n\times n}\right)\\
&=& r^2\left(\left[\begin{array}{cccccc}
	|a_{n-1}|&\frac{1}{2}(1+|a_{n-2}|)&	\frac{1}{2}|a_{n-3}|&\ldots&\frac{1}{2}|a_1|&\frac{1}{2}|a_0|\\
	\frac{1}{2}(1+|a_{n-2}|)&0&\frac{1}{2}&\ldots&0&0\\
	\frac{1}{2}|a_{n-3}|&\frac{1}{2}&0&\ldots&0&0\\
	\vdots&\vdots&\vdots& &\vdots&\vdots\\
	\frac{1}{2}|a_1|&0&0&\ldots&0&\frac{1}{2}\\
	\frac{1}{2}|a_0|&0&0&\ldots&\frac{1}{2}&0\\
\end{array}\right]_{n\times n}\right)\\
&=& w^2\left(\left[\begin{array}{cccccc}
	|a_{n-1}|&\frac{1}{2}(1+|a_{n-2}|)&	\frac{1}{2}|a_{n-3}|&\ldots&\frac{1}{2}|a_1|&\frac{1}{2}|a_0|\\
	\frac{1}{2}(1+|a_{n-2}|)&0&\frac{1}{2}&\ldots&0&0\\
	\frac{1}{2}|a_{n-3}|&\frac{1}{2}&0&\ldots&0&0\\
	\vdots&\vdots&\vdots& &\vdots&\vdots\\
	\frac{1}{2}|a_1|&0&0&\ldots&0&\frac{1}{2}\\
	\frac{1}{2}|a_0|&0&0&\ldots&\frac{1}{2}&0\\
\end{array}\right]_{n\times n}\right)\\
&=& w^2(R+S),
	\end{eqnarray*}
where $$R=\left[\begin{array}{ccccc}
|a_{n-1}|&\frac{1}{2}(1+|a_{n-2}|)&\ldots&\frac{1}{2}|a_1|&\frac{1}{2}|a_0|\\
\frac{1}{2}(1+|a_{n-2}|)&0&\ldots&0&0\\
\frac{1}{2}|a_{n-3}|&0&\ldots&0&0\\
\vdots&\vdots& &\vdots&\vdots\\
\frac{1}{2}|a_0|&0&\ldots&0&0\\
\end{array}\right]_{n\times n}$$ and $$S=\left[\begin{array}{cccccc}
0&0&0	&\ldots&0&0\\
0&0&\frac{1}{2}&\ldots&0&0\\
0&\frac{1}{2}&0&\ldots&0&0\\
\vdots&\vdots&\vdots& &\vdots&\vdots\\
0&0&0&\ldots&0&\frac{1}{2}\\
0&0&0&\ldots&\frac{1}{2}&0\\
\end{array}\right]_{n\times n}.$$
As $R$ and $S$ are Hermitian matrices, it is easy to verify that $$w^2(R+S)\leq w^2(R)+w^2(S)+2w(RS).$$
Simple calculations shows that $$w(R)=r(R)= \frac{1}{2} \left ( |a_{n-1}|+ \sqrt{(1+|a_{n-2}|)^2+ \sum_{j=0,\,j\neq n-2}^{n-1}|a_j|^2}\right).$$ Also,  by  Lemmas \ref{lem10} and \ref{lem3} we get, $$w(S)=r(S)=\cos \frac{\pi}{n}.$$
Now since, $$RS=\frac{1}{4}\left[\begin{array}{ccccccc}
0&|a_{n-3}|&1+|a_{n-2}|+|a_{n-4}|&|a_{n-3}|+|a_{n-5}| &\ldots&|a_2|+||a_0|& |a_1|\\
0&0&0&0&\ldots&0&0\\	
\vdots&\vdots&\vdots&\vdots& &\vdots&\vdots\\
0&0&0&0&\ldots&0&0\\
\end{array}\right]_{n\times n},$$
so by Lemma \ref{lem4} we have, 
\[w(RS)=\frac{1}{8}\sqrt{|a_{n-3}|^2+(1+|a_{n-2}|+|a_{n-4}|)^2+ \sum_{j=0}^{n-4}(|a_{j+1}|+|a_{j-1}|)^2}.\]
Therefore,
\begin{eqnarray*}
	w^2(C(p))&\leq& \cos^2 \frac{\pi}{n}+ \frac{1}{4} \left ( |a_{n-1}|+ \sqrt{(1+|a_{n-2}|)^2+ \sum_{j=0,\,j\neq n-2}^{n-1}|a_j|^2}\right)^2\\
	&& + \frac{1}{4}\sqrt{|a_{n-3}|^2+(1+|a_{n-2}|+|a_{n-4}|)^2+ \sum_{j=0}^{n-4}(|a_{j+1}|+|a_{j-1}|)^2}.
\end{eqnarray*}
	Since, $|z|\leq r(C(p))\leq w(C(p))$, so we have the desired inequality.
\end{proof}

Next we obtain the following bound.

\begin{theorem}\label{th3}
	If $z$ is any zero of $p$, then
\[|z|^2 \leq \cos^2 \frac{\pi}{n+1}+|a_{n-2}|+ \frac{1}{4} \left ( |a_{n-1}|+ \sqrt{ \sum_{j=0}^{n-1}|a_j|^2}\right)^2 +  \frac{1}{2}\sqrt{\sum_{j=0}^{n-2}|a_j|^2}+  \frac{1}{2}\sqrt{ \sum_{j=0}^{n-1}|a_j|^2}.\]


\end{theorem}

\begin{proof}
	Let  $$R=\left[\begin{array}{ccccc}
	-a_{n-1}&-a_{n-2}&\ldots&-a_1&-a_0\\
	0&0&\ldots&0&0\\
	0&0&\ldots&0&0\\
	\vdots&\vdots& &\vdots&\vdots\\
	0&0&\ldots&0&0\\
	\end{array}\right]_{n\times n} \,\text{and} \,\, S=\left[\begin{array}{ccccc}
	0&0&\ldots&0&0\\
	1&0&\ldots&0&0\\
	0&1&\ldots&0&0\\
	\vdots&\vdots& &\vdots&\vdots\\
	0&0&\ldots&1&0\\
	\end{array}\right]_{n\times n}.$$ Then, $C(p)=R+S$. It is easy to verify that $$r^2(R+S)=r\left((R+S)^2\right)\leq w\left((R+S)^2\right)\leq w^2(R)+w^2(S)+w(RS)+w(SR).$$ By Lemma \ref{lem3}, we have $$w(S)=\cos \frac{\pi}{n+1}$$  and by Lemma \ref{lem4}, we have $$w(R)= \frac{1}{2} \left ( |a_{n-1}|+ \sqrt{ \sum_{j=0}^{n-1}|a_j|^2}\right).$$ Since, $$RS=\left[\begin{array}{ccccc}
	-a_{n-2}&-a_{n-3}&\ldots&-a_0&0\\
	0&0&\ldots&0&0\\
	0&0&\ldots&0&0\\
	\vdots&\vdots& &\vdots&\vdots\\
	0&0&\ldots&0&0\\
	\end{array}\right]_{n\times n},$$ so by Lemma \ref{lem4}, we have $$w(RS)= \frac{1}{2} \left ( |a_{n-2}|+ \sqrt{ \sum_{j=0}^{n-2}|a_j|^2}\right).$$ Now, $$SR=\left[\begin{array}{ccccc}
	0&0&\ldots&0&0\\
	-a_{n-1}&-a_{n-2}&\ldots&-a_1&-a_0\\
	0&0&\ldots&0&0\\
	\vdots&\vdots& &\vdots&\vdots\\
	0&0&\ldots&0&0\\
	\end{array}\right]_{n\times n}. $$
	Let $U$ be an unitary matrix obtained by interchanging first and second row of the identity matrix $I_{n}$, i.e., $$U=\left[\begin{array}{cccccc}
	0&1&0&\ldots&0&0\\
	1&0&0&\ldots&0&0\\
	0&0&1&\ldots&0&0\\
	\vdots&\vdots&\vdots&  &\vdots&\vdots\\
	0&0&0&\ldots&0&1\\
	\end{array}\right]_{n\times n}.$$ Then we have, $$U^*SRU=\left[\begin{array}{ccccc}
	-a_{n-2}&-a_{n-1}&\ldots&-a_1&-a_0\\
	0&0&\ldots&0&0\\
	0&0&\ldots&0&0\\
	\vdots&\vdots& &\vdots&\vdots\\
	0&0&\ldots&0&0\\
	\end{array}\right]_{n\times n}.$$
	Therefore, by weak unitary invariance property of the numerical radius, i.e., $w(U^*SRU) $ $ =w(SR)$, and using Lemma \ref{lem4} we get $$w(SR)= \frac{1}{2} \left ( |a_{n-2}|+ \sqrt{ \sum_{j=0}^{n-1}|a_j|^2}\right).$$
	Thus, 
	\begin{eqnarray*}
	r^2(C(p))&\leq & \cos^2 \frac{\pi}{n+1}+|a_{n-2}|+ \frac{1}{4} \left ( |a_{n-1}|+ \sqrt{ \sum_{j=0}^{n-1}|a_j|^2}\right)^2 +  \frac{1}{2}\sqrt{\sum_{j=0}^{n-2}|a_j|^2}\\
	&&+  \frac{1}{2}\sqrt{\sum_{j=0}^{n-1}|a_j|^2}.
	\end{eqnarray*}
	Since,  $|z|\leq r(C(p))$, so we have the desired bound.
	
 \end{proof}

\begin{remark}
In  \cite[Th. 2.10]{OK}, Abu-Omar and Kittaneh proved that if $z$ is any zero of $p$, then
\[|z|^2 \leq \cos^2 \frac{\pi}{n+1}+ \frac{1}{4} \left ( |a_{n-1}|+ \sqrt{ \sum_{j=0}^{n-1}|a_j|^2}\right)^2 +  \sqrt{\sum_{j=0}^{n-1}|a_j|^2}.\]
The comparison of \cite[Th. 2.10]{OK} and the Fujii and Kubo's bound (\ref{zero4}) is also given in \cite{OK}.
Clearly, our bound in Theorem \ref{th3} is sharper than \cite[Th. 2.10]{OK}  if and only if $$2|a_{n-2}|< \sqrt{\sum_{j=0}^{n-1}|a_j|^2}-\sqrt{\sum_{j=0}^{n-2}|a_j|^2}. $$
\end{remark}

\noindent Next, using the  two bounds for the moduli of the zeros of $p$ obtained in Theorem \ref{th2} and Theorem \ref{th3}, we develope  two different bounds. First we assume that $q(z)=(z-a_{n-1})p(z)$, i.e.,  $$q(z)=z^{n+1}-b_{n-1}z^{n-1}-b_{n-2}z^{n-2}-\ldots-b_1z-b_0,$$ where $b_j=a_{n-1}a_j-a_{j-1}$, $j=0,1,2,\ldots,n-1$, $a_{-1}=0.$ Then the Frobenius companion matrix $C(q)$, associated with $q$, is given by
$$C(q)=\left[\begin{array}{cccccc}
0&b_{n-1}&b_{n-2}&\ldots&b_1&b_0\\
1&0&0&\ldots&0&0\\
0&1&0&\ldots&0&0\\
\vdots&\vdots&\vdots& &\vdots&\vdots\\
0&0&0&\ldots&1&0\\
\end{array}\right]_{n+1\times n+1}.$$
Clearly, if $z_1,z_2,\ldots,z_n$ are the zeros of $p$, then $a_{n-1},z_1,z_2,\ldots,z_n$ are the zeros of $q$.
Now using the Frobenius companion matrix  $C(q)$ and applying Theorem \ref{th2}, we have the following bound for the moduli of the zeros of $p$.

\begin{cor}\label{th5}
	If $z$ is any zero of $p$, then
	\begin{eqnarray*}
		|z|^2&\leq& \cos^2 \frac{\pi}{n+1}+\frac{1}{4}  \left( {(1+|b_{n-1}|)^2+\sum_{j=0}^{n-2}|b_j|^2} \right)\\
		&&+ \frac{1}{4}\sqrt{|b_{n-2}|^2+(1+|b_{n-1}|+|b_{n-3}|)^2+ \sum_{j=0}^{n-3}(|b_{j+1}|+|b_{j-1}|)^2},
	\end{eqnarray*}
	where $b_j=a_{n-1}a_j-a_{j-1}$, $j=0,1,2,\ldots,n-1$, $a_{-1}=0$ and $b_{-1}=0$.
\end{cor}

Next by using the Frobenius companion matrix  $C(q)$ and applying Theorem \ref{th3}, we have the following  bound for the moduli of the zeros of $p$.

\begin{cor}\label{th6}
	If $z$ is any zero of $p$, then
	\begin{eqnarray*}
		|z|^2 &\leq& \cos^2 \frac{\pi}{n+2}+|b_{n-1}|+ \frac{1}{4}  \sum_{j=0}^{n-1}|b_j|^2 +  \sqrt{\sum_{j=0}^{n-1}|b_j|^2},
	\end{eqnarray*}
	where $b_j=a_{n-1}a_j-a_{j-1}$, $j=0,1,2,\ldots,n-1$, $a_{-1}=0.$
\end{cor}

\begin{remark}
Consider $p(z)=z^3+z^2+2.$ Simple calculations show that the bounds in  (\ref{zero6}), (\ref{zero5}), (\ref{zero4}) and (\ref{zero1}) are not uniformly better than our new bounds obtained in Theorems \ref{th1}, \ref{th8}. \ref{th2}, \ref{th3} and Corollary \ref{th5}.	
In particular, the bound in Theorem \ref{th8} is much smaller than the existing bounds in  (\ref{zero6}),  (\ref{zero5}), (\ref{zero4}), (\ref{zero1})  (\ref{zero2}) and (\ref{zero3}).  
 \end{remark}

Clearly, the zeros of the polynomial $\frac{z^n}{a_0}p(\frac{1}{z})$ are the reciprocal of the zeros of $p$. Therefore, lower bounds for the zeros of $p$ can be obtained  by considering the polynomial $\frac{z^n}{a_0}p(\frac{1}{z})$ and using  Theorems \ref{th1},  \ref{th7},  \ref{th8}, \ref{th2}, \ref{th3} and Corollaries \ref{th5}, \ref{th6}. This enables us to describe annuli in the complex plane containing all the zeros of $p$. In the following theorem we state one such results obtained by using  Theorem \ref{th8}.
\begin{theorem}\label{lower1}
If $z$ is any zero of $p$, then
\begin{eqnarray}
	|z|^2&\geq& \beta,
\end{eqnarray}
where 	\begin{eqnarray*}
	\frac{1}{\beta }&=& \cos^2 \frac{\pi}{n}+ \frac{1}{4} \left ( |d_{n-1}|+ \sqrt{(1+|d_{n-2}|)^2+ \sum_{j=0,\,j\neq n-2}^{n-1}|d_j|^2}\right)^2\\
	&&+ \frac{1}{2}  \sqrt{(1+|d_{n-4}|)^2+ \sum_{j=0,\,j\neq n-4}^{n-3}|d_j|^2}
\end{eqnarray*}
and $d_j=\frac{a_{n-j}}{a_0},\, j=0,1,\ldots,n-1$,  $a_n=1.$  $($If $n=3$ then $d_{-1}=0$$)$.
\end{theorem}

\begin{remark}
Considering $p(z)=z^3+z^2+z+1$ it is easy to see that  Theorem \ref{th8} and  Theorem \ref{lower1} give smaller annulus containing all the zeros of $p(z)$ than the existing ones obtained by Dalal and Govil \cite{DG1} and Kim \cite{Kim}.
\end{remark}

Finally, we obtain a rectangular region in the complex plane, which contains all the zeros of $p$. For this first we need the following lemma.

\begin{lemma}\label{lem5}
	Let $A\in M_n(\mathbb{C})$. Then $$W(A)\subseteq W(\Re(A))+\rm i W(\Im(A)),$$ where $\Re(A)=\frac{A+A^*}{2}$ and $\Im(A)=\frac{A-A^*}{2\rm i}.$ 
\end{lemma}

\begin{theorem}\label{th4}
	If $z$ is any zero of $p$, then $$z\in [-\mu_1, \mu_1] \times [-\mu_2, \mu_2],$$ where 
	\[\mu_1=\cos\frac{\pi}{n}+\frac{1}{2}\left(|\text{Re}(a_{n-1})|+\sqrt{|\text{Re}(a_{n-1})|^2+|1-a_{n-2}|^2+\sum_{j=0}^{n-3}|a_j|^2}\right)\] and 
	\[\mu_2=\cos\frac{\pi}{n}+\frac{1}{2}\left(|\text{Im}(a_{n-1})|+\sqrt{|\text{Im}(a_{n-1})|^2+|1+a_{n-2}|^2+\sum_{j=0}^{n-3}|a_j|^2}\right).\]
\end{theorem}
 
 \begin{proof}
 Let $$R=\left[\begin{array}{cccccc}
 		-\text{Re}(a_{n-1})&\frac{1}{2}(1-a_{n-2})&-\frac{1}{2}{a_{n-3}}&\ldots&-\frac{1}{2}a_1&-\frac{1}{2}a_0\\
 		\frac{1}{2}(1-\overline{a_{n-2}})&0&0&\ldots&0&0\\
 		-\frac{1}{2}\overline{a_{n-3}}&0&0&\ldots&0&0\\
 		\vdots&\vdots&\vdots& &\vdots&\vdots\\
 		-\frac{1}{2}\overline{a_0}&0&0&\ldots&0&0\\
 	\end{array}\right]_{n\times n}$$ and $$S=\left[\begin{array}{cccccc}
 		0&0&0	&\ldots&0&0\\
 		0&0&\frac{1}{2}&\ldots&0&0\\
 		0&\frac{1}{2}&0&\ldots&0&0\\
 		\vdots&\vdots&\vdots& &\vdots&\vdots\\
 		0&0&0&\ldots&0&\frac{1}{2}\\
 		0&0&0&\ldots&\frac{1}{2}&0\\
 	\end{array}\right]_{n\times n}.$$ Then, $$\Re(C(p))=R+S.$$ Since, $\Re(C(p)$, $R$ and $S$ are Hermitian matrices, so we get $$r\left(\Re(C(p)\right))=w\left(\Re(C(p)\right))\leq w(R)+w(S)=r(R)+r(S).$$ By using Lemma \ref{lem10} and Lemma \ref{lem3}, we have $$r(S)=\cos \frac{\pi}{n}.$$ Also, by simple calculations, we get $$r(R)=\frac{1}{2}\left(|\text{Re}(a_{n-1})|+\sqrt{|\text{Re}(a_{n-1})|^2+|1-a_{n-2}|^2+\sum_{j=0}^{n-3}|a_j|^2}\right).$$ Hence, 
 	\[r\left(\Re(C(p)\right)) \leq \cos\frac{\pi}{n}+\frac{1}{2}\left(|\text{Re}(a_{n-1})|+\sqrt{|\text{Re}(a_{n-1})|^2+|1-a_{n-2}|^2+\sum_{j=0}^{n-3}|a_j|^2}\right).\]
 	Again by similar arguments as above we have,
 	\[r\left(\Im(C(p)\right)) \leq \cos\frac{\pi}{n}+\frac{1}{2}\left(|\text{Im}(a_{n-1})|+\sqrt{|\text{Im}(a_{n-1})|^2+|1+a_{n-2}|^2+\sum_{j=0}^{n-3}|a_j|^2}\right).\]
 	Clearly, $$|Re(z)|\leq r\left(\Re(C(p)\right))$$ and $$|Im(z)|\leq r\left(\Im(C(p)\right)).$$ Also,  Lemma \ref{lem5} implies that  $$z\in \left[\lambda_{\min}(\Re(C(p))),\lambda_{\max}(\Re(C(p)))  \right] \times  \left[\lambda_{\min}(\Im(C(p))),\lambda_{\max}(\Im(C(p)))  \right].$$  This completes the proof.
 \end{proof}

\bibliographystyle{amsplain}

\end{document}